\documentclass[12pt]{amsart}
\usepackage{euscript,amssymb,xr}

\let\oldmarginpar\marginpar
\renewcommand\marginpar[1]
{\oldmarginpar{\tiny\bf \begin{flushleft} #1 \end{flushleft}}}

\input xy
\xyoption{all}

\newcommand{\la}{\langle}
\newcommand{\ra}{\rangle}

\newtheorem{theorem}{\bf Theorem}[section]
\newtheorem{lemma}[theorem]{\bf Lemma}

\newtheorem{corollary}[theorem]{\bf Corollary}

\newcommand{\FF}{{\Bbb F}}

\newcommand{\NN}{{\Bbb N}}

\newcommand{\QQ}{{\Bbb Q}}
\newcommand{\RR}{{\Bbb R}}

\newcommand{\ZZ}{{\Bbb Z}}

\newcommand{\ggreat}{>\kern-.7ex>}
\newcommand{\ssmall}{<\kern-.7ex<}
\newcommand{\qu}{/\kern-.7ex/}
\newcommand{\exh}{\to\kern-1.8ex\to}

\newcommand{\cC}{{\EuScript{C}}}

\newcommand{\fF}{{\EuScript{F}}}
\newcommand{\gG}{{\EuScript{G}}}

\newcommand{\pP}{{\EuScript{P}}}

\newcommand{\GL}{\operatorname{GL}}

\newcommand{\Diff}{\operatorname{Diff}}

\newcommand{\Ham}{\operatorname{Ham}}

\newcommand{\Homeo}{\operatorname{Homeo}}
\newcommand{\Id}{\operatorname{Id}}

\newcommand{\Ker}{\operatorname{Ker}}

\newcommand{\rk}{\operatorname{rk}}

\newcommand{\Stab}{\operatorname{Stab}}

\newcommand{\discsym}{\operatorname{disc-sym}}

\title[Jordan property and almost fixed point property]
{Jordan property for homeomorphism groups and almost fixed point property}
\author{Ignasi Mundet i Riera}
\address{Facultat de Matem\`atiques i Inform\`atica\\
Universitat de Barcelona\\
Gran Via de les Corts Catalanes 585\\
08007 Barcelona \\
Spain}
\email{ignasi.mundet@ub.edu}

\date{\today}

\subjclass[2010]{57S17,54H15}

\thanks{This research was partially supported by the grant
PID2019-104047GB-I00 from the Spanish Ministeri de Ci\`encia i Innovaci\'o.}

\begin{document}

\maketitle

\begin{abstract}
We study properties of continuous finite group actions on topological manifolds that
hold true, for any finite group action, after possibly passing to a subgroup
of index bounded above by a constant depending only on the manifold. These include
the Jordan property, the almost fixed point property, as well as bounds
on the discrete symmetry group. Most of our results apply to manifolds satisfying
some restriction such as having nonzero Euler characteristic or having the
integral homology of a sphere. For an arbitrary
topological manifold $X$ such that $H_*(X;\ZZ)$ is finitely generated, we prove
the existence a constant $C$ with the property that for any continuous action of
a finite group $G$ on $X$ such that every $g\in G$ fixes at least on
point of $X$, there is a subgroup $H\leq G$ satisfying $[G:H]\leq C$ and
a point $x\in X$ which is fixed by all elements of $H$.
\end{abstract}

\section{Introduction}

\subsection{Main results}
\label{ss:main-results}
Our aim in this paper is to prove several results on continuous finite group actions
on topological manifolds which, despite not being necessarily true for all group actions, are valid
up to passing to subgroups of uniformly bounded index. This includes for example results on the Jordan
property of homeomorphisms group, as we will explain below, but we
will also consider a few other properties. Given the need to pass to subgroups of bounded index,
our results are especially meaningful when considering continuous actions of large finite groups.

To materialize the previous idea we introduce some terminology.
Let $\pP$ and $\pP'$ be two general properties of continuous finite group actions on topological manifolds.
Both $\pP$ and $\pP'$ may refer to the algebraic structure of the finite group  that acts
(it may be abelian, nilpotent, a $p$-group for an arbitrary prime $p$, etc.),
to the geometry of the action (being effective, free, with
or without fixed points, or, if the manifold has the necessary additional structure,
being smooth, symplectic, complex, etc.), or to both the group and the action.

Let $X$ be a topological manifold. We say that {\it almost every} finite group action on $X$ that enjoys
some property $\pP$ satisfies also property $\pP'$ if there exists a constant $C$, depending only on $X$,
such that the following is true:
\begin{quote}
for every action of a finite group $G$ on $X$ which satisfies
property $\pP$ there exists a subgroup $G'\leq G$ such that $[G:G']\leq C$ and such that the action of $G'$ on
$X$ (defined by restricting the action of $G$) satisfies property $\pP'$.
\end{quote}
If property $\pP'$ refers only to the group and not to the action, then we will say that almost every group
acting on $X$ with property $\pP$ satisfies also property $\pP'$.

Here is a simple example of this notion, which is \cite[Lemma 2.6]{M2} with cohomology replaced by homology
(the proof is identical).

\begin{lemma}
\label{lemma:trivial-homology}
Let $X$ be a topological manifold such that $H_*(X;\ZZ)$ is finitely generated.
Almost every continuous finite group action on $X$ induces the trivial action
on the integral homology $H_*(X;\ZZ)$.
\end{lemma}

We next state the main results in this paper.
From now one all finite group actions on topological manifolds will be implicitly assumed to be continuous.
Topological manifolds will {\it not} be implicitly assumed to be compact in this paper,
and they may have nonempty boundary. As usual, a closed manifold means a compact manifold without boundary.

\begin{theorem}
\label{thm:low-dimension}
Let $X$ be a connected and compact topological manifold
of dimension at most $3$.
Almost every finite group acting effectively on $X$ is abelian.
\end{theorem}

For any real number $a$ we denote by $[a]$ the biggest integer smaller than or equal to $a$.
For any topological manifold $X$ we denote as usual by $H_*(X;\ZZ)=\bigoplus_{k\geq 0}H_k(X;\ZZ)$ its integral homology.
If $H_*(X;\ZZ)$ is finitely generated, the Euler
characteristic of $X$ is defined to be $\chi(X):=\sum_{k\geq 0}(-1)^k\dim_{\QQ}H_k(X;\ZZ)\otimes_{\ZZ}\QQ$.
If $X$ is a compact topological manifold then $H_*(X;\ZZ)$ is finitely generated, but this is not always
true if $X$ is not compact.

\begin{theorem}
\label{thm:nonzero-Euler-Jordan}
Let $X$ be a connected $n$-dimensional topological manifold with
$H_*(X;\ZZ)$ finitely generated and such that $\chi(X)\neq 0$.
Almost every finite group acting on $X$ is abelian and can be generated by $[n/2]$ or fewer elements.
\end{theorem}

\begin{theorem}
\label{thm:homology-sphere}
Let $X$ be an $n$-dimensional topological manifold such that $H_*(X;\ZZ)\simeq H_*(S^n;\ZZ)$.
Almost every finite group acting effectively on $X$ is abelian and can be generated
by $[(n+1)/2]$ or fewer elements.
\end{theorem}

As usual, when we say that the action of a group $G$ on $X$ has a fixed point we mean that there exists a point $x\in X$ fixed by all elements of $G$.

\begin{theorem}
\label{thm:compact-nonzero-Euler}
Let $X$ be a connected and compact topological manifold
with nonzero Euler characteristic.
Almost every finite group action on $X$ has a fixed point.
\end{theorem}

In the previous theorem, compactness is essential.
This follows for example from the main result in \cite{HKMS}
(see the end of \cite[\S 1.1]{M6} for a detailed explanation).

Let us say that an action of a group $G$ on a manifold $X$ has the {\it weak fixed point property}
if for every $g\in G$ there is some $x\in X$ such that $g\cdot x=x$. Here the point
$x$ may depend on $g$, so a priori the weak fixed point property does not imply the existence
of a fixed point for the entire action. The next theorem implies that, nevertheless, for finite group
actions this turns out to be the case, up to passing to a subgroup of controlled index.

\begin{theorem}
\label{thm:from-weak-to-almost-fixed-point}
Let $X$ be a connected $n$-dimensional topological manifold with finitely generated $H_*(X;\ZZ)$.
For almost every action of a finite group $G$ on $X$ with the weak
fixed point property, the group $G$ is abelian, it can be generated by
$[n/2]$ or fewer elements, and its action on $X$ has a fixed point.
\end{theorem}

We next explain some application of the previous theorem to symplectic geometry.
Let $(X,\omega)$ be a symplectic manifold. The group
of Hamiltonian diffeomorphisms of $X$, which we denote as usual by $\Ham(X,\omega)$,
consists of those diffeomorphisms $\phi\in\Diff(X)$ for which there exists a family
of diffeomorphisms $\{\phi_t\in\Diff(X)\}_{t\in[0,1]}$, smoothly
depending on $t$, satisfying
$\phi_0=\Id_X$,
$\phi_1=\phi$, and the following property:
\begin{quote}
for every $t\in [0,1]$ there exists some $H_t\in\cC^{\infty}(X)$
such that, for every vector field $\fF$ on $X$, the equality
$\omega(\partial_t\phi_t,\fF)=dH_t(\fF)$ holds.
\end{quote}
These properties imply in particular that $\phi$ is a symplectomorphism.
Choosing $\{\phi_t\}$ and $\{H_t\}$ appropriately we may assume that
$\{H_t\}$ extends to a $\ZZ$-periodic
smooth function of $t\in \RR$ (see e.g. \cite[\S 4.1]{AbbSchlenk}).
Suppose that $X$ is compact.
A standard argument implies that any $\phi\in\Ham(X,\omega)$
can be uniformly approximated by Hamiltonian diffeomorhisms all of whose fixed points are
nondegenerate. We may thus apply Arnold's conjecture (which is now a theorem by
\cite{FukayaOno,LiuTian})
to conclude that each $\phi\in\Ham(X,\omega)$ has a fixed point in $X$.
So the Theorem \ref{thm:from-weak-to-almost-fixed-point} implies the following:

\begin{corollary}
\label{cor:Ham}
Let $(X,\omega)$ be a $2m$-dimensional closed symplectic manifold.
For almost every action of a finite group $G$ on $X$
which is defined by a monomorphism $G\to \Ham(X,\omega)$
the group $G$ is abelian, it can be generated by $m$ or fewer elements,
and its action on $X$ has a fixed point.
\end{corollary}

The statement that $G$ is almost always abelian in the previous corollary (equivalently, the
fact that $\Ham(X,\omega)$ is Jordan --- see below) was proved
in \cite{M7} using a different method. The existence of a fixed point
seems to be a new result.

\subsection{Jordan property}
Following Popov \cite{Po0}, we say that a group $\gG$ is Jordan if there exists some constant $C$ such that any
finite subgroup $\Gamma$ of $\gG$ has an abelian subgroup $A\leq\Gamma$ satisfying
$[\Gamma:A]\leq C$.
The most basic nontrivial examples are general
linear groups, for which the Jordan property was proved by C. Jordan in 1878 \cite{J}, using, of course,
a different terminology:

\begin{theorem}
\label{thm:Jordan}
For any $n$, $\GL(n,\RR)$ is Jordan.
\end{theorem}

\'E. Ghys asked around 1990 whether the diffeomorphism group of
any closed smooth manifold is Jordan. A number of papers \cite{M1,M4,Z2} appeared
in the last few years giving partial positive answers to Ghys's
question.  It 2014 B. Csik\'os, L. Pyber and E. Szab\'o \cite{CPS1} found the first
counterexamples to Ghys's question, showing that the
diffeomorphism group of $T^2\times S^2$ is not Jordan
(later, extending the ideas in \cite{CPS1}, the author and D.R. Szab\'o found many other examples,
see \cite{M6,Szabo}).

In contrast to diffeomorphism groups, the Jordan property for homeomorphism groups has received
less attention so far, with the exception of a paper of S. Ye \cite{Ye}, which proves that closed
flat manifolds such as tori have Jordan homeomorphism groups, and \cite{M9}, which proves that
rationally hypertoral manifolds have Jordan homeomorphism groups.
The results stated in Subsection \ref{ss:main-results} imply the following.
\begin{itemize}
\item The homeomorphism group of any closed and connected topological manifold of
dimension at most $3$ is Jordan (Theorem \ref{thm:low-dimension}).
\item Let $X$ be a connected topological manifold. If $H_*(X;\ZZ)$ is
finitely generated and $\chi(X)\neq 0$ then the homeomorphism group of $X$ is Jordan
(Theorem \ref{thm:nonzero-Euler-Jordan}).
\item If $X$ is an $n$-dimensional manifold satisfying $H_*(X;\ZZ)\simeq H_*(S^n;\ZZ)$
then the homeomorphism group of $X$ is Jordan (Theorem \ref{thm:homology-sphere}).
\item For any closed symplectic manifold $(X,\omega)$ the group of Hamiltonian
diffeomorphisms $\Ham(X,\omega)$ is Jordan (Corollary \ref{cor:Ham}).
\end{itemize}
The analogue of the first statement for diffeomorphism groups was proved
by Zimmermann in \cite{Z2}. Analogues of the second and third statements
for diffeomorphism groups have been proved in \cite{M4}.

In view of \cite{CPS1}, Ghys asked whether for any closed manifold $X$ almost every
finite group acting effectively on $X$ is nilpotent.
This has been recently proved by
B. Csik\'os, L. Pyber and E. Szab\'o \cite{CPS2} in full generality for continuous actions
on topological manifolds.

\subsection{Almost fixed point property}

We say that a topological
manifold $X$ has the almost fixed point property if there exists a constant $C$ such that for any
action of a finite group $G$ on $X$ there exists a point $x\in X$ whose stabilizer
$G_x=\{g\in G\mid g\cdot x=x\}$ satisfies $[G:G_x]\leq C$.
In this terminology, we may rephrase Theorem \ref{thm:compact-nonzero-Euler} as follows.
\begin{itemize}
\item Let $X$ be connected, compact, topological manifold with
nonzero Euler characteristic; then $X$ has the almost fixed point property.
\end{itemize}
This generalizes the main result in \cite{M8} in two directions. First, it replaces
smooth actions by continuous actions. Second, and more importantly, it replaces the hypothesis
of not having odd cohomology (see \cite{M8} for a precise definition) by the much weaker condition
of having nonzero Euler characteristic.

A simple but already interesting example
is the case of a closed disk. Let $n$ be a natural number and suppose
that a finite group $G$ acts on the closed $n$-disk $D^n$ by homeomorphisms. By Brouwer's fixed point theorem, for every $g\in G$ there exists at least one point in $D^n$ which is fixed by the action of $g$. A priori such fixed point depends on $g$.
If $n\leq 4$ then one can actually pick some $x\in D^n$ which is simultaneously fixed by all elements of $G$, but if $n\geq 6$ this is
no longer true (see the introduction in \cite{M8} for references). However, Theorem \ref{thm:compact-nonzero-Euler} implies the existence,
for each $n$, of a constant $\lambda\in (0,1]$, depending only on $n$,
with the property that for every action of a finite group $G$ on $D^n$ there exists some
point in $D^n$ which is fixed by at least $\lambda\cdot|G|$ elements of $G$.

\subsection{Discrete degree of symmetry}

A celebrated result of Mann and Su \cite{MS} states that for any closed manifold
$X$ there exists a constant $M$ with the following property: for any prime $p$ and
any natural number $m$ such that $(\ZZ/p)^m$ admits an effective action on $X$
we have $m\leq M$. This result is also true for manifolds with boundary and
whose integral homology is finitely generated (see Section \ref{ss:some-tools}). It follows that the set
$$\mu(X)=\{m\in\NN\mid \text{$X$ supports effective actions of $(\ZZ/r)^m$ for arbitrarily large $r$}\}$$
is finite: more precisely, $\mu(X)$ is contained in $\{1,\dots,M\}$
(of course, $\mu(X)$ may be empty). Following \cite{M9} we define
the discrete degree of symmetry of $X$ to be
$$\discsym(X)=\max (\{0\}\cup\mu(X)).$$
By \cite[Lemma 1.10]{M9}, for any nonnegative integer $k$ the inequality $\discsym(X)\leq k$
is equivalent to the statement that almost every finite abelian group acting effectively on
$X$ can be generated by $k$ or fewer elements (if $k=0$ the latter means by convention that
the abelian group is trivial). Consequently, Theorems \ref{thm:nonzero-Euler-Jordan}
and \ref{thm:homology-sphere} have the following implications, respectively:
\begin{itemize}
\item Let $X$ be an $n$-dimensional connected topological manifold. If $H_*(X;\ZZ)$ is finitely generated and $\chi(X)\neq 0$ then $\discsym(X)\leq [n/2]$.
\item Let $X$ be an $n$-dimensional topological manifold such that $H_*(X;\ZZ)\simeq H_*(S^n;\ZZ)$; then $\discsym(X)\leq [n/2]$.
\end{itemize}
It was asked at the end of \cite[\S 1.4]{M9} whether for any compact and connected $n$-dimensional manifold $X$ we have $\discsym(X)\leq n$.
This was motivated by the well known fact that if $X$ supports a continuous and effective action of a torus
$(S^1)^d$ then $d\leq n$, and the heuristic according to which the sequence of finite groups $(\ZZ/r)^m$, with $m$ fixed and $r\to\infty$,
can be thought of as increasingly good approximations of the torus $T^m$. But note that this heuristic has its limitations:
as explained in \cite[Theorem 1.14]{M9}, not every compact and connected manifold $X$ supports an effective and continuous action of $(S^1)^{\discsym(X)}$. Some evidence for the conjectural inequality $\discsym(X)\leq\dim X$
was given in \cite{M9}, and the previous results give some additional evidence.

\subsection{Some tools}
\label{ss:some-tools}
As usual, by a finite $p$-group we mean a finite group whose cardinal is a power of an arbitrary prime number $p$.

We next state some results that will play a key role in the proofs of our theorems.
While all the results that we mention in this section are originally proved only
for manifolds without boundary, all of them extend to the case of manifolds with
boundary thanks to the following result.

\begin{lemma}
\label{lemma:taking-boundary-away}
Let $X$ be a topological manifold. There is a topological manifold $X'$
without boundary, of the same dimension as $X$, and an embedding $X\hookrightarrow X'$
which is a homotopy equivalence and which has the property that any
action of a group $G$ on $X$ extends to an action on $X'$.
\end{lemma}
\begin{proof}
Let $U=\partial X\times [0,1)$ and let $X'=(X\sqcup U)/\sim$, where $\sim$ identifies
each $x\in\partial X$ with $(x,0)\in U$. The natural inclusion $X\hookrightarrow X'$
is clearly a homotopy equivalence. Given an action of a group $G$ on $X$ we extend
it to an action on $X'$ by declaring the action of $g\in G$ on $(x,t)\in U$ to be
$(g\cdot x,t)$.
\end{proof}

The first result, which is a particular case of \cite[Theorem 1.8]{CMPS},
is an extension of the theorem of Mann and Su \cite{MS}
to non necessarily compact manifolds.

\begin{theorem}
\label{thm:MS-non-compact}
Let $X$ be a topological manifold. If $H_*(X;\ZZ)$ is finitely generated
then there exists a constant $M$ with this property: for any prime $p$ and
any natural number $m$ such that $(\ZZ/p)^m$ admits an effective action on $X$
we have $m\leq M$.
\end{theorem}

The next theorem is based on results from \cite{CPS2}.

\begin{theorem}
\label{thm:CPS-abstract}
Let $X$ be a topological manifold. Suppose that $H_*(X;\ZZ)$ is finitely generated.
For every positive number $C$ there exists a positive number $C'$ with the
following property. Let $G$ be a finite group acting effectively on $X$.
Suppose that for every prime $p$ and any Sylow $p$-subgroup $G_p$ of $G$
there is an abelian subgroup $G_p^a\leq G_p$ satisfying $[G_p:G_p^a]\leq C$.
Then there is an abelian subgroup $G^a\leq G$ satisfying $[G:G^a]\leq C'$.
\end{theorem}
\begin{proof}
Either combine Theorem \ref{thm:MS-non-compact}, \cite[Corollary 3.18]{CPS2}, \cite[Lemma 6.1]{CPS2}
and \cite[Lemma 5.3]{CPS2}, or alternatively Theorem \ref{thm:MS-non-compact}, \cite[Lemma 6.1]{CPS2},
\cite[Lemma 5.3]{CPS2} and \cite[Theorem 3.8]{MT}.
\end{proof}

The following consequence is immediate.

\begin{corollary}
\label{cor:CPS}
Let $X$ be a topological manifold such that $H_*(X;\ZZ)$ is finitely generated.
Suppose that almost every finite $p$-group that acts effectively on $X$
is abelian. Then almost every finite group that acts effectively on $X$
is abelian.
\end{corollary}

If a group $G$ acts on a manifold $X$ we denote the set of stabilizers of the action as
$$\Stab(G,X)=\{G_x\mid x\in X\}.$$
The following is part of \cite[Theorem 1.3]{CMPS}.

\begin{theorem}
\label{thm:CMPS}
Let $X$ be a topological manifold. If $H_*(X;\ZZ)$ is finitely generated then
there exists a constant $C$ such that for almost every action of
a finite $p$-group $G$ on $X$ we have $|\Stab(G,X)|<C.$
\end{theorem}

\subsection{Contents} Section \ref{s:almost-fixed-points} contains sobre
results on almost fixed points, in Section \ref{s:DH} we prove how the
existence of a fixed point for a finite $p$-group action on an $n$-manifold
implies that  the group can be embedded in $\GL(n,\RR)$,
and finally in Section \ref{s:proof} we prove the
theorems stated above.

\subsection{Acknowledgements} The author is very pleased to thank L. Pyber for sending him
a copy of the paper \cite{CPS2}, on which most of the
results in this paper are based. Many thanks also to E. Szab\'o for useful correspondence.

\section{Almost fixed points}
\label{s:almost-fixed-points}

From now on, cohomology will implicitly refer to
Alexander--Spanier cohomology. This is canonically isomorphic to singular cohomology
for topological manifolds, but it is better behaved than singular homology when
considering fixed point sets of finite $p$-group actions on topological manifolds.
Similarly, Betti numbers are to be defined as dimensions of Alexander--Spanier cohomology
groups, and the Euler characteristic is, when defined, the alternate sum of the dimensions of
Alexander--Spanier cohomology groups.

\begin{lemma}
\label{lemma:p-group-nonzero-Euler-characteristic}
Let $X$ be a topological manifold. If $H_*(X;\ZZ)$ is finitely generated and
$\chi(X)\neq 0$ then almost every finite $p$-group action on $X$ has a fixed point.
\end{lemma}
\begin{proof}
Let $p$ be a prime and let $G$ be a finite $p$-group
acting continuously and effectively on $X$. Let the order of $G$ be $p^m$.
For any nonnegative integer $i$
let $X_i$ be the set of points in $X$ whose stabilizer has order $p^i$.
We may view $H^j(X_i;\ZZ/p)$ as a vector space over the field $\ZZ/p$,
and we accordingly define $b_j(X_i;\ZZ/p)=\dim_{\ZZ/p} H^j(X_i;\ZZ/p)$.
By \cite[Theorem 2.5]{Ye2}
we have
$b_j(X_i;\ZZ/p)<\infty$ for each $i,j$; furthermore, $b_j(X_i;\ZZ/p)=0$ for big enough $j$,
and $\chi(X_i)=\sum_{j\geq 0} (-1)^jb_j(X_i;\ZZ/p)$ is divisible by $p^{m-i}$ for every $i$;
finally, $\chi(X)=\sum_{i\geq 0} \chi(X_i)$. Hence, if we let $p^i$ be the largest power of $p$
that divides $\chi(X)$, which of course satisfies $p^i\leq |\chi(X)|$, we necessarily have $X_{m-k}\neq \emptyset$ for some $k\leq i$. It follows that there exists a subgroup $G'\leq G$ satisfying $[G:G']\leq |\chi(X)|$ and
$X^{G'}\neq\emptyset$.
\end{proof}

The following is an analogue for continuous finite $p$-group actions on topological manifolds
of \cite[Lemma 9]{M8}, which refers to smooth group actions on smooth manifolds.
For the definition of $\ZZ/p$-cohomology manifold, see \cite[Chap I]{Bo}.

\begin{lemma}
\label{lemma:descending-chain}
Let $X$ be a topological manifold without boundary.
If $H_*(X;\ZZ)$ is finitely generated then here exists a constant $D$ with the following
property. Let $p$ be any prime. Suppose given a chain of strict inclusions of $\ZZ/p$-cohomology submanifolds
$$X_1\subsetneq X_2\subsetneq\dots\subsetneq X_r\subseteq X,$$
where for each $j$ there is a finite $p$-group $G_j$ acting continuously on $X$ in such a way that
$X_j$ is the union of some of the connected components of $X^{G_j}$. Then
$r\leq D$.
\end{lemma}
\begin{proof}
For any action of a finite $p$-group $G$ on a $\ZZ/p$-cohomology manifold $X$ we have
\begin{equation*}
|\pi_0(X^G)|=\dim_{\FF_p}H^0(X^{G};\FF_p)\leq \sum_j\dim_{\FF_p}H^j(X^{G};\FF_p)\leq \sum_j\dim_{\FF_p}H^j(X;\FF_p).
\end{equation*}
This follows from \cite[Theorem III.4.3]{Bo} and an easy induction on the cardinal of $G$
(see e.g. \cite[Lemma 5.1]{M7}).
The lemma is a consequence of this estimate on $|\pi_0(X^G)|$ and
the arguments in the proof of \cite[Lemma 9]{M8}, together with basic properties
of the dimension function for cohomology submanifolds (namely, that dimension is well
defined for connected $\ZZ/p$-cohomology manifolds and that it decreases when passing
from a connected $\ZZ/p$-cohomology
manifold to a proper connected $\ZZ/p$-cohomology submanifold).
\end{proof}

Given an action of a group $G$ on a set $X$ we denote for every $g\in G$
$$X^g=\{x\in X\mid g\cdot x=x\}$$
the set of points in $X$ fixed by $g$.

\begin{theorem}
\label{thm:generic-point}
Let $X$ be a topological manifold.
Suppose that $H_*(X;\ZZ)$ is finitely generated.
There exists a constant $L$
with the following property. Let $G$ be a
finite $p$-group acting on $X$. There exists a
subgroup $K\leq G$ and an element $g\in K$ satisfying $[G:K]\leq L$ and
$X^{K}=X^g$.
\end{theorem}

To avoid confusion, the reader should keep in mind that the previous theorem
does not rule out the possibility that $X^K=\emptyset$.

\begin{proof}
By Lemma \ref{lemma:taking-boundary-away} it suffices to consider the
case in which $X$ has no boundary. So we assume in the remainder of the proof that
this is the case.
By Theorem \ref{thm:CMPS} there exist numbers $C,C'$, depending only on $X$,
such that for any finite $p$-group $G$ acting on $X$
there exists a subgroup $G_s\leq G$ satisfying $|\Stab(G_s,X)|<C$ and $[G:G_s]\leq C'$.
Let $D$ be the number given by applying Lemma \ref{lemma:descending-chain} to
the manifold $X$. We claim that $L=(CC')^DC'$ satisfies the property given in the
statement.

Let $G$ be a finite $p$-group acting on $X$.
We claim that there is a subgroup $G'\leq G$ satisfying $[G:G']\leq (CC')^D$
and which has the property that any subgroup $H\leq G'$ such that
$X^{G'}\subsetneq X^H$ has index $[G':H]>CC'$. Indeed, if no such
subgroup $G'$ existed, then we could construct a sequence of subgroups
$$G=:G_0>G_1>G_2>\dots>G_D$$ satisfying
$X^{G_i}\subsetneq X^{G_{i+1}}$
and $[G_{i}:G_{i+1}]\leq CC'$ for each $i\geq 0$. This would contradict
Lemma \ref{lemma:descending-chain}, so the claim is proved.

By Theorem \ref{thm:CMPS} there exists a subgroup $K\leq G'$
satisfying $[G':K]\leq C'$ and $|\Stab(K,X)|< C$.
The first property implies that
$$[G:K]=[G:G']\cdot[G':K]\leq (CC')^DC'=L.$$
We next prove that there exists an element $g\in K$ such that $X^K=X^g$.

Let $S$ be the collection of all $H\in\Stab(K,X)$ such that
$H\neq K$. For any $H\in\Stab(K,X)$ the condition
$H\neq K$ is equivalent to $X^{K}\subsetneq X^H$.
By our choice of $K$ the set
$S$ contains at most $|\Stab(K,X)|<C$ elements.

Let
$H\in S$. There exists some $x\in X$ such that $H=K_x$
because $H\in \Stab(K,X)$. We have
$K_x=G_x'\cap K$, so $G_x'\neq G'$ (for otherwise $H=K_x$
would be equal to $K$). Hence $G_x'\in\Stab(G',X)$ satisfies
$X^{G'}\subsetneq X^{G_x'}$ and so, by the choice of $G'$, we
have $[G':G_x']\geq CC'$. The bound $[G':K]\leq C'$ implies that
$|K|\geq |G'|/C'$, so
$$[K:H]=[K:K_x]=[K:G'_x\cap K]=\frac{|K|}{|G'_x\cap K|}\geq \frac{|G'|/C'}{|G'_x|}=\frac{[G':G'_x]}{C'}
\geq \frac{CC'}{C'}=C.$$
Hence for every $H\in S$ we have $|H|\leq |K|/C$

We have
$$\left|\bigcup_{H\in S}H\right|\leq \sum_{H\in S}|H|\leq |S|\frac{|K|}{C}\leq |\Stab(K,X)|\frac{|K|}{C}<|K|.$$
So we may take some element $g\in K$ that does not belong to any
$H\in S$. Certainly $X^{K}\subseteq X^g$. If the inclusion were strict,
then we could take some $x\in X^g\setminus X^{K}$.
Then $K_x$ would belong to $S$ and would contain $g$, which is a
contradiction. Consequently, $X^{K}=X^g$.
\end{proof}

\begin{theorem}
\label{thm:fixed-point-abelian-group-action-weak-fixed-point}
Let $X$ be a topological manifold. If $H_*(X;\ZZ)$ is finitely generated then
almost every finite abelian group action on $X$ with the weak fixed point property
has a fixed point.
\end{theorem}
\begin{proof}
Let $L$ be the number given by applying Theorem \ref{thm:generic-point} to $X$.
We are going to prove that for any action of a finite abelian group $A$ on $X$ with the
weak fixed point property there is
a subgroup $A'\leq A$ satisfying $[A:A']\leq L^L$ and $X^{A'}\neq\emptyset$.

Let us fix a finite abelian group $A$ and suppose that $A$ acts on $X$ with the weak
fixed point property. We have $A=\prod_p A_p$, where the product runs over the
set of primes $p$ and $A_p$ is the $p$-part of $A$.
In other words, $A_p$ is the Sylow $p$-subgroup of $A$ or, equivalently (since $A$ is abelian),
the group consisting of elements whose order is a power of $p$.
By Theorem \ref{thm:generic-point} for any $p$ there is a subgroup $A_p'\leq A_p$ and an element
$a_p\in A_p'$ such that $X^{a_p}=X^{A_p'}$, and furthermore $[A_p:A_p']\leq L$.
The latter implies that $A_p'=A_p$ for every $p>L$, since for such values of $p$ no finite
$p$-group can contain a proper subgroup of index not greater than $L$.
Since the number of primes in $\{1,\dots,L\}$ is obviously at most $L$, the group
$A'=\prod_pA_p'$ satisfies $[A:A']\leq L^L$.

Let $a=\prod_p a_p\in A'$.
Since the action of $A$ on $X$ has the weak fixed point property,
$X^a\neq\emptyset$.
Since the order of each $a_p$ is a power of $p$, the Chinese remainder theorem implies
that for each $p$ there is some $e$ such that $a^e=a_p$. This implies
that $X^a\subseteq X^{a_p}=X^{A_p'}$ for each prime $p$. Consequently
$$\emptyset\neq X^a\subseteq \bigcap_p X^{A_p'}=X^{A'},$$
so the proof of the theorem is now complete.
\end{proof}

\begin{theorem}
\label{thm:fixed-point-abelian-group-action}
Let $X$ be a compact connected manifold
with nonzero Euler characteristic.
Almost every finite abelian group action on $X$ has a fixed point.
\end{theorem}
\begin{proof}
By Lemma \ref{lemma:trivial-homology} there exists a natural number $H$,
depending only on $X$, such
that for any action of a finite group $G$ on $X$ there is a subgroup
$G'\leq G$ whose action on $H_*(X;\ZZ)$ is trivial and which satisfies
$[G:G']\leq H$. Since $X$ is compact, $H_*(X;\ZZ)$ is finitely generated.
Let $H'$ be the constant obtained by applying
Theorem \ref{thm:fixed-point-abelian-group-action-weak-fixed-point}
to $X$. We are going to prove that for any action of a finite abelian group $A$
on $X$ there is a subgroup $A''\leq A$ satisfying $X^{A''}\neq\emptyset$
and $[A:A'']\leq HH'$.

Let $A$ be a finite abelian group acting on $X$.
There is a subgroup $A'\leq A$ whose
action on $H_*(X;\ZZ)$ is trivial and which satisfies $[A:A']\leq H$.
Since $X$ is compact and has nonzero Euler characteristic,
Lefschetz's formula \cite[Exercise 6.17.3]{tD} implies that
the action of $A'$ on $X$ has the weak fixed point property.
Applying Theorem \ref{thm:fixed-point-abelian-group-action-weak-fixed-point}
to the action of $A'$ on $X$ we conclude that there is a subgroup
$A''\leq A'$ satisfying $X^{A''}\neq\emptyset$
and $[A':A'']\leq H'$.
Now, $[A:A'']\leq HH'$.
\end{proof}

\section{Linearising actions of finite $p$-groups at fixed points}

\label{s:DH}

It is well known that if a compact Lie group $G$ acts smoothly and
effectively on a connected smooth
$n$-dimensional manifold $X$ and $x\in X^G$ is a fixed point, then
the linearization of the action provides an effective linear action of $G$ on $T_xX$.
This implies that $G$ is isomorphic to a subgroup of $\GL(n,\RR)$.

The last property is false for continuous actions. Indeed,
Zimmermann has given in \cite{Z3} an example, for each natural
number $n>5$, of a finite group
which acts effectively on $S^n$ and yet is not isomorphic
to any subgroup of $\GL(n+1,\RR)$. Any such action induces an
effective action on the
cone $CS^n=(S^n\times [0,\infty))/(S^n\times\{0\})\cong\RR^{n+1}$
fixing the point that arises from collapsing
$S^n\times\{0\}$ (i.e., the vertex of the cone).

In this section we prove two main results. The first one, Corollary \ref{cor:DH-local-manifold-with-boundary},
states that if we restrict to actions of finite $p$-groups then the previous pathology
cannot take place. The second one, Theorem \ref{thm:DH-sphere},
states that any finite $p$-group acting
continuously on an $n$-dimensional manifold whose integral homology
is isomorphic to $H_*(S^n;\ZZ)$ is isomorphic to a subgroup of
$\GL(n+1,\RR)$ (in contrast to the case of arbitrary finite group actions,
again by Zimmermann \cite{Z3}).

The proofs of these results follow in a
straightforward way from the ideas in a paper of Dotzel and Hamrick \cite{DH}
and from standard results on continuous actions
of finite $p$-groups on topological manifolds mainly due to Smith and Borel.

The following lemma is a consequence of the arguments in \cite[Section 1]{DH} (see also
\cite[Chap III, Theorem 5.13]{tD}).

\begin{lemma}
\label{lemma:DH}
Let $p$ be a prime and $G$ be a finite $p$-group.
Suppose given a function assigning a nonnegative integer
$n(H)$ to each subgroup $H\leq G$, satisfying these
properties:
\begin{enumerate}
\item For any two subgroups $K\lhd K'\leq G$ such that
$K'/K$ is elementary abelian of rank $2$ the following holds:
$$n(K)-n(K')=\sum_H (n(H)-n(K')),$$
where the sum is over the subgroups $H\leq K'$
satisfying $K\lhd H$ and $H/K\simeq\ZZ/p$;
\item For any two subgroups $K\leq K'\leq G$, we have
$n(K')\leq n(K)$, and if $p$ is odd then $n(K)-n(K')$ is even.
\item For any subgroup $K\leq G$ and any $g\in G$ we have
$n(K)=n(gKg^{-1})$.
\item For any three subgroups $K\lhd K'\lhd K''\leq G$ such that
$K'/K\simeq\ZZ/2$ and $K''/K$ is generalized quaternion
(resp. cyclic of order $4$), $n(K)-n(K')$ is divisible by $4$
(resp. $n(K)-n(K')$ is even).
\end{enumerate}
If $n(K)<n:=n(\{1\})$ for every subgroup $K\leq G$ different from $\{1\}$,
then $G$ is isomorphic to a subgroup of $\GL(n,\RR)$.
\end{lemma}

Conditions (1)---(4) above are usually called Borel--Smith conditions.
The result proved in \cite[Section 1]{DH} and in
\cite[Chap III, Theorem 5.13]{tD} states that given any function
$G\geq H\mapsto n(H)\geq 0$ satisfying the Borel--Smith conditions there
is a real representation $\rho:G\to\GL(V)$ with the property that
$\dim V^H=n(H)$ for every $H\leq G$. This implies in particular that $\dim V=\dim V^{\{1\}}=n:=n(1)$
and that, if $n(K)<n$ for every subgroup $K\leq G$ different from $\{1\}$,
then the action of $G$ on $V$ is effective. Indeed, for any $g\in G$
we have $\Ker(\rho(g)-\Id)=V^{\la g\ra}$, where $\la g\ra\leq G$ denotes
the subgroup generated by $g$, and if $g\neq 1$ then $\dim V^{\la g\ra}<\dim V$,
so $\rho(g)\neq 1$.

\begin{theorem}
\label{thm:DH-local}
Let $p$ be a prime, let $G$ be a finite $p$-group, and suppose that
$G$ acts effectively on an $n$-dimensional connected
$\ZZ/p$-cohomology manifold $X$ without boundary.
If $X^G\neq\emptyset$ then $G$ is isomorphic to
a subgroup of $\GL(n,\RR)$.
\end{theorem}
\begin{proof}
Take any $x\in X^G$. For any $H\leq G$ the fixed point set $X^H$ is a
$\ZZ/p$-cohomology submanifold of $X$ containing $x$ (see \cite[Chap I]{Bo} for the
definition of $\ZZ/p$-cohomology subma\-nifold and \cite[Chap V, Theorem 2.2]{Bo} for the proof
of the statement on $X^H$).
Define $n(H):=\dim_x X^H$.
We next explain why the function $G\geq H\mapsto n(H)$ satisfies the Borel--Smith
conditions.

Condition (1) is proved in \cite[Chap XIII, Theorem 4.3]{Bo}.
The first part of condition (2) is a consequence of the fact, previously
used in the proof of Lemma \ref{lemma:descending-chain} above, that the dimension
of cohomology manifolds is nonincreasing when passing to a submanifold (this follows
from the general properties of dimension of cohomology manifolds, see \cite[Chap I]{Bo}).
The second part of condition (2) is proved in \cite[Chap V, Theorem 2.3, (a)]{Bo}.
Condition (3) follows from $X^{gKg^{-1}}=gX^K$ and the homeomorphism invariance of
dimension.
Finally, condition (4) follows from the same arguments as Propositions (4.31) and (4.32)
in \cite[Chap III]{tD}.

Since the action of $G$ on $X$ is effective and $X$ is connected, we have
$n(K)<n$ for every subgroup $K\leq G$ different from $\{1\}$.
Consequently, the theorem follows from Lemma \ref{lemma:DH}.
\end{proof}

\begin{corollary}
\label{cor:DH-local-manifold-with-boundary}
Let $G$ be a finite $p$-group, and suppose that
$G$ acts effectively on an $n$-dimensional connected
topological manifold $X$. If $X^G\neq\emptyset$ then $G$ is isomorphic to
a subgroup of $\GL(n,\RR)$.
\end{corollary}
\begin{proof}
If $X$ has no boundary, then Theorem \ref{thm:DH-local} applies to $X$ for every prime $p$,
so there is nothing to be proved. Suppose that $\partial X\neq\emptyset$.
Let $f:X\hookrightarrow X'$ be the embedding given by applying Lemma \ref{lemma:taking-boundary-away}
to $X$. Then $X'$ is an $n$-dimensional topological manifold.
Since $f$ is a homotopy equivalence and $X$ is connected, $X'$ is also connected.
In addition, for every prime $p$, $X'$ is a $\ZZ/p$-cohomology  manifold without boundary.
Let $G$ be a finite $p$-group acting effectively on $X$, and suppose that $X^G\neq\emptyset$.
Consider the extension of this action to $X'$ given by Lemma \ref{lemma:taking-boundary-away}.
Then $(X')^G\neq\emptyset$. Applying Theorem \ref{thm:DH-local} to the action of $G$ on $X'$
we conclude that $G$ is isomorphic to a subgroup of $\GL(n,\RR)$.
\end{proof}

\begin{theorem}
\label{thm:DH-sphere}
Let $p$ be a prime, let $G$ be a finite $p$-group, and suppose that
$G$ acts effectively on an $n$-dimensional $\ZZ/p$-homology sphere.
Then $G$ is isomorphic to a subgroup of $\GL(n,\RR)$.
\end{theorem}
\begin{proof}
It follows from applying Theorem \ref{thm:DH-local} to
the induced action on the cone $CX=X\times [0,\infty)/(X\times\{0\})$,
which is a $\ZZ/p$-homology manifold. Alternatively, one may use results
for actions of finite $p$-groups on $\ZZ/p$-homology spheres analogous
to the results on actions with fixed points that we used in the
proof of Theorem \ref{thm:DH-local} (see \cite{Bo,tD}).
\end{proof}

\section{Proofs of the theorems}
\label{s:proof}

\subsection{Proof of Theorem \ref{thm:low-dimension}: manifolds of dimension at most $3$}
\label{ss:low-dimension}

If $X$ is a compact manifold with nonempty boundary we denote by $X^{\sharp}$ the double
of $X$, resulting from taking the disjoint union of two copies of $X$ and identifying the
boundary of the first copy with the boundary of the second copy using the tautological
identification between them.
Then $X^{\sharp}$ is a closed
manifold of the same dimension as $X$, and any group acting effectively on $X$ also acts effectively
on $X^{\sharp}$. Consequently, to prove Theorem \ref{thm:low-dimension} it suffices to consider
closed manifolds.

Assume that
$X$ is a closed topological manifold of dimension $3$. By Moise's theorem \cite{Moise}
(see also \cite{Bing2}), $X$ has a unique smooth structure (actually both Moise and Bing refer
to triangulations in their papers; for the existence of a unique smooth structure on triangulated
manifolds of dimensions at most $3$ see \cite[Section 3.10]{Thurston}).
By a recent result of Pardon \cite{Pardon2019}
any finite group acting effectively and topologically on $X$ admits effective smooth actions on $X$ (although, of course, not any topological action is conjugate to a smooth action, as illustrated by the famous example due to Bing \cite{Bing}).
The previous facts also hold true in dimensions $1$ and $2$ with substantially simpler proofs.
In dimension $1$ they are an easy exercise.
See \cite{Rado} for the $2$-dimensional version of Moise's theorem (Rad\'o's theorem), and
\cite[pp. 340--341]{Edmonds} and the references therein for the other statements.

Since the diffeomorphism group $\Diff(X)$ is known to be Jordan if $X$ is closed and $\dim X\leq 3$
(see \cite{M1} for the case $\dim X\leq 2$ and
Zimmermann \cite{Z2} for the case $\dim X=3$), it follows that $\Homeo(X)$ is Jordan as well,
so Theorem \ref{thm:low-dimension} is proved.

\subsection{Proof of Theorem
\ref{thm:nonzero-Euler-Jordan}}
\label{ss:proof-thm:nonzero-Euler-Jordan}
Let $X$ be a connected $n$-dimensional topological manifold
with finitely generated integral homology and nonzero Euler characteristic.

We first prove that almost every finite group acting effectively on $X$ is abelian.
By Corollary \ref{cor:CPS}, it suffices to prove that
almost every finite $p$-group acting effectively on $X$ is abelian.
By Lemma \ref{lemma:p-group-nonzero-Euler-characteristic}
almost every finite $p$-group action on $X$ has a fixed point.
Let $G$ be a finite $p$-group acting effectively on $X$ and suppose that
$X^G\neq\emptyset$. By Corollary \ref{cor:DH-local-manifold-with-boundary}, $G$ is isomorphic
to a subgroup of $\GL(n,\RR)$, and by Theorem \ref{thm:Jordan}
$G$ has an abelian subgroup $G'\leq G$ with $[G:G']$ bounded
above by a constant depending only on $n$. This concludes the proof
that almost every finite group acting effectively on $X$ is abelian.

We next prove that almost every abelian group acting effectively on $X$
can be generated by $[n/2]$ or fewer elements. For any finite abelian
group $A$ let $\rk A$ denote the size of the smallest generating subsets
of $A$. For any prime $p$ denote by $A_p$ the $p$-part of $A$. We
have $\rk A=\max_p\rk A_p$, where the maximum is taken over the set
of all primes $p$. So to prove that
almost every abelian group acting effectively on $X$
can be generated by $[n/2]$ or fewer elements,
it suffices to prove that
almost every abelian $p$-group acting effectively on $X$
can be generated by $[n/2]$ or fewer elements.
The arguments used above to prove the Jordan property for $\Homeo(X)$
imply that it suffices to prove that almost every finite abelian $p$-subgroup
of $\GL(n,\RR)$ can be generated by $[n/2]$ or fewer elements, which
is an easy and well known fact.

\subsection{Proof of Theorem \ref{thm:homology-sphere}}
The proof is almost identical to that of Theorem \ref{thm:nonzero-Euler-Jordan},
except that to prove that almost every finite $p$-group acting on $X$ is abelian
we use Theorems \ref{thm:DH-sphere} and \ref{thm:Jordan}.

\subsection{Proof of Theorem \ref{thm:compact-nonzero-Euler}}
Let $X$ be a compact and connected $n$-dimensional topological manifold
with nonzero Euler characteristic.
It clearly suffices to consider effective group actions.
By Theorem \ref{thm:nonzero-Euler-Jordan}, almost every finite group acting effectively
on $X$ is abelian. Consequently, we only need to prove that almost every
finite abelian group acting effectively on $X$ has a fixed point. This is
precisely the statement of Theorem \ref{thm:fixed-point-abelian-group-action}.

\subsection{Proof of Theorem \ref{thm:from-weak-to-almost-fixed-point}}
Let $X$ be a connected topological manifold with finitely generated
integral homology.
Combining Theorem \ref{thm:generic-point},
Corollary \ref{cor:DH-local-manifold-with-boundary},
and Theorem \ref{thm:Jordan},
we conclude that there is a constant $C$ such that, for every prime $p$, any
finite $p$-group $P$ acting effectively and with the weak fixed point property on $X$
has an abelian subgroup $P'\leq P$ satisfying $[P:P']\leq C$.
Applying Theorem \ref{thm:CPS-abstract} we conclude the existence of a constant $C'$ such that
any finite group $G$ acting effectively and with the weak fixed point property on $X$ has
an abelian subgroup $A\leq G$ satisfying $[G:A]\leq C'$.
By Theorem \ref{thm:fixed-point-abelian-group-action-weak-fixed-point}
there is a constant $C''$ depending only on $X$
and a subgroup $B\leq A$ such that $X^B\neq\emptyset$
and $[A:B]\leq C''$. Since $[G:B]\leq C'C''$ the proof of the theorem is complete.


\begin{thebibliography}{99}

\bibitem{AbbSchlenk}
A. Abbondandolo, F. Schlenk,
Floer homologies, with applications,
{\em Jahresber. Dtsch. Math.-Ver.} {\bf 121} (2019), no. 3, 155--238.


\bibitem{Bing}
R.H. Bing,
A homeomorphism between the 3-sphere and the sum of two solid horned spheres.
{\em Ann. of Math.} (2) {\bf 56} (1952), 354--362.

\bibitem{Bing2}
R.H. Bing,
An alternative proof that 3-manifolds can be triangulated,
{\em Ann. of Math.} (2) {\bf 69} (1959), 37--65.

\bibitem{Bo} A. Borel,
{\em Seminar on transformation groups}, Ann. of Math. Studies {\bf 46},
Princeton University Press, N.J., 1960.

\bibitem{CPS1} B. Csik\'os, L. Pyber, E. Szab\'o, Diffeomorphism groups of compact 4-manifolds
are not always Jordan, {\em preprint} {\tt arXiv:1411.7524}.

\bibitem{CPS2} B. Csik\'os, L. Pyber, E. Szab\'o, Finite subgroups of the homeomorphism group
of a compact manifold are almost nilpotent, {\em preprint} {\tt arXiv:2204.13375}.

\bibitem{CMPS} B. Csik\'os, I. Mundet i Riera, L. Pyber, E. Szab\'o, On the number
of stabilizer subgroups in a finite group acting on a manifold, {\em preprint} {\tt arXiv:2111.14450}.

\bibitem{DH}
R.M. Dotzel, G.C. Hamrick,
$p$-group actions on homology spheres,
{\em Invent. Math.} {\bf 62} (1981) 437--442.


\bibitem{Edmonds}
A.L. Edmonds, Transformation groups and low-dimensional manifolds, Group
actions on manifolds (Boulder, Colo., 1983), Contemp. Math., vol. 36, Amer. Math.
Soc., Providence, RI, 1985, pp. 339--366.


\bibitem{FukayaOno}
K. Fukaya, K. Ono,
Arnold conjecture and Gromov-Witten invariant,
{\em Topology} {\bf 38} (1999), no. 5, 933--1048.

\bibitem{HKMS}
R. Haynes, S. Kwasik, J. Mast, R. Schultz,
{\em Periodic maps on $\RR^7$ without fixed points},
Math. Proc. Cambridge Philos. Soc. {\bf 132} (2002), no. 1, 131--136.

\bibitem{J}
C. Jordan, M\'emoire sur les \'equations diff\'erentielles lin\'eaires \`a int\'egrale alg\'ebrique,
{\em J. Reine Angew. Math.} {\bf 84} (1878) 89--215.

\bibitem{LiuTian}
G. Liu, G. Tian,
Floer homology and Arnold conjecture,
{\em J. Differential Geom.} {\bf 49} (1998), no. 1, 1--74.

\bibitem{MS}
L. N. Mann, J. C. Su, Actions of elementary p-groups on manifolds,
{\em Trans. Amer. Math. Soc.} {\bf 106} (1963), 115--126.

\bibitem{Moise}
E.E. Moise,
Affine structures in 3-manifolds. V. The triangulation theorem and Hauptvermutung,
{\em Ann. of Math.} (2) {\bf 56} (1952), 96--114.

\bibitem{M1} I. Mundet i Riera, Jordan's theorem for the
    diffeomorphism group of some manifolds {\em Proc. Amer.
    Math. Soc.} {\bf 138} (2010), 2253-2262.

\bibitem{M2} I. Mundet i Riera,
Finite group actions on 4-manifolds with nonzero Euler characteristic,
{\em Math. Z.} {\bf 282} (2016), no. 1--2, 25--42.


\bibitem{M6} I. Mundet i Riera,
Non Jordan groups of diffeomorphisms and actions of compact Lie groups on manifolds,
{\it Transformation Groups} {\bf 22} (2017), no. 2, 487--501,
DOI 10.1007/s00031-016-9374-9.

\bibitem{M7} I. Mundet i Riera,
Finite subgroups of Ham and Symp.,
{\em Math. Ann.} {\bf 370} (2018), no. 1--2, 331--380.


\bibitem{M4} I. Mundet i Riera,
    Finite group actions on homology spheres and manifolds
    with nonzero Euler characteristic,
    {\em J. Topology} {\bf 12} (2019), 743--757.

\bibitem{M8} I. Mundet i Riera,
Almost fixed points of finite group actions on manifolds without odd cohomology,
{\em Transform. Groups} {\bf 25} (2020), no. 4, 1269--1288.

\bibitem{M9} I. Mundet i Riera,
Topological rigidity of tori and finite group actions on hypertoral manifolds,
{\em preprint} {\tt arXiv:2112.05599}.


\bibitem{MT}
I. Mundet i Riera, A. Turull,
Boosting an analogue of Jordan's theorem for finite groups,
{\em Adv. Math.} {\bf 272} (2015), 820--836.


\bibitem{Pardon2019}
J. Pardon, Smoothing finite group actions on three-manifolds,
{\em Duke Math. J.} {\bf 170} (2021), no. 6, 1043--1084.

\bibitem{Po0} V.L. Popov, On the Makar-Limanov, Derksen
    invariants, and finite automorphism groups of algebraic
    varieties. In {\em Peter Russell’s Festschrift, Proceedings of
    the conference on Affine Algebraic Geometry held in
    Professor Russell’s honour}, 1–5 June 2009, McGill Univ.,
    Montreal., volume 54 of Centre de Recherches Math\'ematiques
    CRM Proc. and Lect. Notes, pages 289–311, 2011.

\bibitem{Rado}
T. Rad\'o, \"Uber den Begrieff der Riemannschen Fl\"ache, {\em Acta Szeged} {\bf 2} (1926) 101--121.

\bibitem{Szabo}
D.R. Szab\'o, Special $p$-groups acting on compact manifolds,
{\it preprint} {\tt arXiv:1901.07319}.

\bibitem{tD}
T. tom Dieck, {\em Transformation Groups}, De Gruyter Studies in Mathematics, vol. {\bf 8},
Walter de Gruyter \& Co., Berlin, 1987.

\bibitem{Thurston}
W.P. Thurston,
Three-dimensional geometry and topology. Vol. 1.
Edited by Silvio Levy. Princeton Mathematical Series, {\bf 35}. Princeton University Press, Princeton, NJ, 1997.

\bibitem{Ye}
S. Ye, Symmetries of flat manifolds, Jordan property and the general Zimmer program,
    {\em preprint} {\tt  arXiv:1704.03580}.

\bibitem{Ye2}
S. Ye, Euler characteristics and actions of automorphism groups of free groups,
{\em Algebr. Geom. Topol.} {\bf 18} (2018), 1195--1204.

\bibitem{Z2}
B.P. Zimmermann, On Jordan
    type bounds for finite
    groups acting on compact $3$-manifolds,
    {\em Arch. Math.} {\bf 103} (2014), 195--200.

\bibitem{Z3}
B. Zimmermann,
On topological actions of finite, non-standard groups on spheres,
{\em Monatsh. Math.} {\bf 183} (2017), no. 1, 219--223.
\end{thebibliography}
\end{document}